\documentclass[11pt]{article}

\usepackage{amsmath,amsthm,amssymb,fancyhdr,graphicx,bbm,cancel,color,mathrsfs,todonotes,hyperref,xypic, pinlabel}
\usepackage[all]{xy}

\usepackage{bm}

\usepackage{lipsum}

\let\OLDthebibliography\thebibliography
\renewcommand\thebibliography[1]{
  \OLDthebibliography{#1}
  \setlength{\parskip}{0pt}
  \setlength{\itemsep}{0pt plus 0.3ex}
}

\newtheorem{thm}{Theorem}
\newtheorem{mainthm}{Theorem}

\newtheorem{lem}[thm]{Lemma}
\newtheorem{prop}[thm]{Proposition}

\theoremstyle{definition} 
\newtheorem{defn}[thm]{Definition}
 
\newtheorem{example}[thm]{Example}
 
\newtheorem{qu}[thm]{Question} 
 
\newtheorem{rmk}[thm]{Remark}
\theoremstyle{remark} 

\newcommand{\dmo}{\DeclareMathOperator}

\newcommand{\R}{\mathbb{R}}
\newcommand{\Q}{\mathbb{Q}}
\newcommand{\co}{\mathbb{C}}\newcommand{\Z}{\mathbb{Z}}

\newcommand{\al}{\alpha}\newcommand{\be}{\beta}\newcommand{\ga}{\gamma}\newcommand{\ep}{\epsilon}\newcommand{\ze}{\zeta}
\newcommand{\Om}{\Omega}\newcommand{\la}{\lambda}
\newcommand{\Ga}{\Gamma}
\newcommand{\wtil}{\widetilde}\newcommand{\Si}{\Sigma}\newcommand{\Lam}{\Lambda}

\newcommand{\cd}{\cdots}
\newcommand{\sbs}{\subset}

\newcommand{\xra}{\xrightarrow}
\newcommand{\car}{\curvearrowright}
\newcommand{\ra}{\rightarrow}
\newcommand{\hra}{\hookrightarrow}\newcommand{\onto}{\twoheadrightarrow}

\newcommand{\bb}[1]{\mathbb{#1}}\newcommand{\ca}[1]{\mathcal{#1}}\newcommand{\un}[1]{\underline{#1}}\newcommand{\mf}{\mathfrak}

\newcommand{\fr}[2]{\frac{#1}{#2}}

\newcommand{\rt}{\rtimes}\newcommand{\ot}{\otimes}

\newcommand{\ti}{\times}

\newcommand{\pair}[1]{\langle #1 \rangle}

\dmo{\sgn}{sign}
\dmo{\we}{\wedge}
\dmo{\ind}{ind}\dmo{\Ind}{Ind}
\dmo{\bop}{\bigoplus}\dmo{\pic}{Pic}
\dmo{\coker}{coker}\dmo{\vol}{Vol}\dmo{\gal}{Gal}\dmo{\perm}{Perm}
\dmo{\tor}{Tor}\dmo{\ext}{Ext}\dmo{\Ext}{Ext}
\dmo{\aut}{aut}
\dmo{\Aut}{Aut}
\dmo{\inn}{Inn}\dmo{\var}{Var}
\dmo{\dep}{depth}
\dmo{\ad}{ad}\dmo{\curl}{curl}
\dmo{\hy}{\bb H}\dmo{\Sl}{SL}
\dmo{\SO}{SO}\dmo{\psl}{PSL}
\dmo{\isom}{Isom}\dmo{\Isom}{Isom}
\dmo{\conf}{Conf}
\dmo{\stab}{Stab}\dmo{\Jac}{Jac }
\dmo{\diam}{diam}\dmo{\fix}{Fixed}\dmo{\Fix}{Fix}
\dmo{\injR}{injRad}\dmo{\Ad}{Ad}
\dmo{\esv}{ess-vol}\dmo{\out}{Out}\dmo{\Out}{Out}
\dmo{\nil}{Nil}\dmo{\sol}{Sol}
\dmo{\Div}{div}
\dmo{\SU}{SU}
\dmo{\SP}{SP}
\dmo{\Sp}{Sp}
\dmo{\rk}{rk}
\dmo{\rank}{rank}
\dmo{\psp}{PSp}\dmo{\psu}{PSU}
\dmo{\PU}{PU}\dmo{\pgl}{PGL}
\dmo{\Mod}{Mod}\dmo{\range}{Range}
\dmo{\eu}{eu}\dmo{\mi}{mi}
\dmo{\Log}{Log}\dmo{\supp}{supp}
\dmo{\maps}{Maps}\dmo{\Gr}{Gr}
\dmo{\Pin}{Pin}
\dmo{\Spin}{Spin}\dmo{\Str}{Str}
\dmo{\Sq}{Sq}\dmo{\Symp}{Symp}
\dmo{\pd}{PD}\dmo{\PD}{PD}\dmo{\sig}{Sig}
\dmo{\Set}{Set}\dmo{\Top}{Top}
\dmo{\ev}{ev}\dmo{\St}{St}
\dmo{\Pt}{Pt}\dmo{\pt}{pt}

\dmo{\colim}{colim }\dmo{\Pl}{PL}
\dmo{\String}{String}\dmo{\smear}{smear}

\dmo{\dev}{dev}
\dmo{\met}{Met}\dmo{\contact}{Contact}
\dmo{\teich}{Teich}\dmo{\Teich}{Teich}\dmo{\qi}{QI}
\dmo{\der}{Der}
\dmo{\cl}{Cliff}\dmo{\Cl}{Cl}
\dmo{\Pf}{Pf}
\dmo{\GL}{GL}
\dmo{\PSL}{PSL}
\dmo{\ch}{ch}\dmo{\diag}{diag}
\dmo{\grad}{grad}\dmo{\Char}{char}
\dmo{\spec}{Spec}\dmo{\Arg}{Arg}
\dmo{\rad}{rad}\dmo{\im}{Im}
\dmo{\Hom}{Hom}\dmo{\End}{End}
\dmo{\tr}{tr}\dmo{\id}{Id}
\dmo{\gl}{GL}
\dmo{\sym}{Sym}\dmo{\Sym}{Sym}
\dmo{\com}{Comm}
\dmo{\Lk}{Lk}
\dmo{\CAT}{CAT}
\dmo{\Rep}{Rep}
\dmo{\Res}{Res}
\dmo{\Conf}{Conf}
\dmo{\PConf}{PConf}
\dmo{\Push}{Push}
\dmo{\Cont}{Cont}
\dmo{\sm}{\setminus}
\dmo{\vn}{\varnothing}
\dmo{\disk}{\mathbb D}
\dmo{\Trd}{Trd}\dmo{\Mat}{Mat}
\dmo{\Riem}{Riem}
\dmo{\Diffn}{\Diff_0}\dmo{\diff}{diff}
\dmo{\Diff}{Diff}\dmo{\homeo}{Homeo}
\dmo{\Homeo}{Homeo}\dmo{\Fr}{Fr}
\dmo{\rot}{rot}\dmo{\Emb}{Emb}
\dmo{\Ham}{Ham}\dmo{\Met}{Met}
\dmo{\Ein}{Ein}\dmo{\CP}{\co P}
\dmo{\Per}{Per}\dmo{\Ric}{Ric}
\newcommand{\C}{\mathbb C}\dmo{\Nrd}{Nrd}
\dmo{\Comp}{Comp}\dmo{\PSC}{PSC}
\dmo{\Cent}{Cent}\dmo{\Orb}{Orb}
\dmo{\aind}{a-ind}\dmo{\tind}{t-ind}
\dmo{\constant}{constant}
\dmo{\Td}{Td}
\dmo{\LMod}{LMod}
\dmo{\SMod}{SMod}
\dmo{\SDiff}{SDiff}
\dmo{\Br}{Br}
\dmo{\csch}{csch}
\dmo{\triv}{triv}
\dmo{\genus}{genus}
\dmo{\Homeq}{HomEq}
\dmo{\PP}{\mathbb{P}}
\dmo{\U}{U}
\dmo{\Gal}{Gal}
\dmo{\BDiff}{\wtil{\Diff}}
\dmo{\BAut}{\wtil{\Aut}}
\dmo{\Iso}{Iso}
\dmo{\SL}{SL}
\dmo{\codim}{codim}
\dmo{\II}{II}
\dmo{\I}{I}
\dmo{\InjRad}{InjRad}
\dmo{\Inn}{Inn}
\dmo{\sys}{sys}
\dmo{\Comm}{Comm}
\dmo{\PO}{PO}
\dmo{\POm}{P\Om}
\dmo{\ab}{ab}
\dmo{\PSO}{PSO}

\usepackage{enumerate,pdfpages,stmaryrd} 
\usepackage[margin=1.5in]{geometry}

\usepackage{tikz}
\usepackage{dsfont}
\usetikzlibrary{matrix}

\begin{document}

\title{Arithmeticity of groups $\Z^n\rt\Z$}

\author{Bena Tshishiku}



\date{\today}


\maketitle

\begin{abstract}
We study when the group $\Z^n\rt_A\Z$ is arithmetic, where $A\in\GL_n(\Z)$ is hyperbolic and semisimple. We begin by giving a characterization of arithmeticity phrased in the language of algebraic tori, building on work of Grunewald--Platonov. We use this to prove several more concrete results that relate the arithmeticity of $\Z^n\rt_A\Z$ to the reducibility properties of the characteristic polynomial of $A$. Our tools include algebraic tori, representation theory of finite groups, Galois theory, and the inverse Galois problem.
\end{abstract}

\section{Introduction}

This paper focuses on the following question. 

\begin{qu}\label{q:arithmetic}
Fix $A\in\GL_n(\Z)$. When is the semidirect product $\Ga_A:=\Z^n\rt_A\Z$ an arithmetic group? 
\end{qu}

Recall that a group $\Ga$ is called \emph{arithmetic} if it embeds in an algebraic group $G$ defined over $\Q$ with image commensurable to $G(\Z)$. 

{\bf Standing assumption.} We restrict focus to the generic case when $A$ is \emph{hyperbolic} (no eigenvalues on the unit circle) and \emph{semisimple} (diagonalizable over $\C$). 

With the standing assumption, Question \ref{q:arithmetic} can be answered in terms of the eigenvalues of $A$ as follows. Let $\chi(A)=\mu_1^{n_1}\cdots\mu_m^{n_m}$ be the characteristic polynomial, decomposed into irreducible factors over $\Q$. Choose a root $\la_i$ of $\mu_i$, and view it as an element in (the $\Z$-points of) the algebraic torus $R_{\Q(\la_i)/\Q}(G_m)$, where $G_m$ is the multiplicative group and $R_{K/\Q}(\cdot)$ denotes the restriction of scalars. 

\begin{mainthm}[Arithmeticity criterion]\label{thm:arithmeticity}
Fix $A\in\GL_n(\Z)$ hyperbolic and semisimple, with characteristic polynomial $\chi(A)=\mu_1^{n_1}\cd\mu_m^{n_m}$ and eigenvalues $\la_1,\ldots,\la_m$ as above. View $\la=(\la_1,\ldots,\la_m)$ as an element of $T(\Z)$, where $T$ is the algebraic torus $\prod_{i=1}^m R_{\Q(\la_i)/\Q}(G_m)$. Let $S\subset T$ be the Zariski closure of the subgroup generated by $\la$. The following are equivalent.
\begin{enumerate}[(i)]
\item The group $\Z^n\rt_A\Z$ is arithmetic. 
\item The rank of $S(\Z)$ as an abelian group is $1$. 
\end{enumerate}
\end{mainthm}

An important component of the proof of Theorem \ref{thm:arithmeticity} is an argument of Grunewald--Platonov \cite{grunewald-platonov-polycyclic} that addresses the arithmeticity question for $\Ga=\ca O_K\rt_\la\Z$, where $\ca O_K$ is the ring of integers in a number field $K$, and $\Z$ acts on $\ca O_K$ by multiplication by a unit $\la\in\ca O_K^\ti$. To relate the work of Grunewald--Platonov to the general case, we use the conjugacy classification of hyperbolic, semisimple elements of $\GL_n(\Z)$ in terms of number fields and ideal classes. 

While Theorem \ref{thm:arithmeticity} gives a complete answer to Question \ref{q:arithmetic}, from a practical viewpoint it is not completely satisfactory because taking the Zariski closure adds a layer of computational difficulty. It would be better if condition (ii) in Theorem \ref{thm:arithmeticity} were phrased directly in terms of the eigenvalues or characteristic polynomial of $A$. To illustrate this point, the reader might try to use Theorem \ref{thm:arithmeticity} to determine if $\Ga_A$ is arithmetic for the two matrices below (this can be done in an ad hoc way, but we give a more systematic approach below). 
\begin{equation}\label{eqn:matrices}
A_1=\left(\begin{array}{rrrrrr}
0&1&0&2\\
0&0&1&0\\
0&1&0&1\\
1&0&1&0
\end{array}\right)\>\>\>\>\>
A_2=\left(
\begin{array}{rrrrrr}
0&0&0&0&-1\\
1&0&0&0&0\\
0&1&0&0&2\\
0&0&1&0&1\\
0&0&0&1&0
\end{array}
\right)\end{equation}

We prove several theorems that refine Theorem \ref{thm:arithmeticity} in special cases. Our most complete result is when all the eigenvalues of $A$ are real.

\begin{mainthm}[Improved arithmeticity criterion: totally real case]\label{thm:real}
Fix $A\in\GL_n(\Z)$ hyperbolic and semisimple, with $\chi(A)=\mu_1^{n_1}\cdots\mu_m^{n_m}$ and $\la=(\la_1,\ldots,\la_m)$ as in the statement of Theorem \ref{thm:arithmeticity}. Assume that all of the eigenvalues of $A$ are real. Then the following are equivalent.
\begin{enumerate}[(i)]
\item The group $\Z^n\rt_A\Z$ is arithmetic. 
\item After replacing $A$ by $A^k$ for some $k\ge1$, the $\la_i$ are all powers of a unit $\ep\in\ca O_L^\ti$ in a real quadratic extension $L/\Q$. 
\end{enumerate} 
\end{mainthm}

Condition (ii) of Theorem \ref{thm:real} implies in particular that each of the $\mu_i$ have degree 2, and so a power of $A$ is conjugate to a block diagonal matrix with $2\ti2$ blocks. As a simple application, the matrix $A_1$ in (\ref{eqn:matrices}) has  $\chi(A_1)=x^4-4x^2+1$, which is irreducible over $\Q$, but $\chi(A_1^2)=(x^2-4x+1)^2$, so we conclude by Theorem \ref{thm:real} that $\Ga_{A_1}$ is arithmetic. For the other matrix $A_2$ in (\ref{eqn:matrices}), $\chi(A_2)=x^5 - x^3 - 2x^2 + 1$ has non-real roots, so Theorem \ref{thm:real} does not apply to this example. 

The author does not know of an analogue of Theorem \ref{thm:real} when $A\in\GL_n(\Z)$ has complex eigenvalues. However, Theorem \ref{thm:real} motivates Question \ref{q:irreducible} below. Before stating it, we need a definition. 

\begin{defn}\label{defn:irreducible}
We say that $A\in\GL_n(\Z)$ is \emph{irreducible} if its characteristic polynomial $\chi(A)$ is irreducible over $\Q$; otherwise we say $A$ is reducible. We say that $A$ is \emph{fully irreducible} if $A^k$ is irreducible for each $k\ge1$. 
\end{defn}

Fully irreducibility of $A$ implies that $A$ is both semisimple and hyperbolic. Note also that $A$ is reducible if and only if it's conjugate in $\GL_n(\Z)$ to a block diagonal matrix $(A_1,A_2)\in\GL_{n_1}(\Z)\times\GL_{n_2}(\Z)$, where $n_1,n_2\ge1$ (here it is important to remember our standing assumption that $A$ is semisimple). 

According to Theorem \ref{thm:real}, if $A\in\GL_n(\Z)$ is fully irreducible and its eigenvalues are real, then $\Ga_A$ is arithmetic if and only if $n=2$. This points us toward the following question. 

\begin{qu}\label{q:irreducible}For which $n\ge2$, does there exist a fully irreducible $A\in\GL_n(\Z)$ so that $\Ga_A=\Z^n\rt_A\Z$ is arithmetic?
\end{qu}

We develop techniques that address Question \ref{q:irreducible}, and use them to prove the next two theorems, which display contrasting behavior. 

\begin{mainthm}[Fully irreducible, arithmetic examples in high dimension]\label{thm:examples}
For each $k\ge1$, there exists $n\ge k$ and a fully irreducible $A\in\GL_n(\Z)$ so that $\Z^n\rt_A\Z$ is arithmetic. 
\end{mainthm}

\begin{mainthm}[No irreducible, arithmetic examples in prime dimension]\label{thm:prime-dim}
Fix a prime $p\ge5$. There does not exist a hyperbolic, irreducible $A\in\GL_p(\Z)$ so that $\Z^p\rt_A\Z$ is arithmetic. 
\end{mainthm}

For example, $A_2$ in (\ref{eqn:matrices}) is irreducible, so $\Ga_{A_2}$ is not arithmetic by Theorem \ref{thm:prime-dim}. 

For each $n$ as in Theorem \ref{thm:examples}, our proof shows that there are infinitely many commensurability classes of arithmetic groups $\Z^n\rt_A\Z$ with $A\in\GL_n(\Z)$ fully irreducible. Theorem \ref{thm:examples} becomes easier if we replace ``fully irreducible" by ``irreducible"; for example, the matrix $A_1$ in (\ref{eqn:matrices}) is  irreducible but not fully irreducible. 

\begin{rmk} For a lattice $\Ga$ in a real semisimple Lie group $G$, much is known about the arithmeticity question, especially from the work of Margulis on superrigidity \cite{margulis}. Margulis proved that any irreducible lattice is arithmetic if $\rank_\R(G)\ge2$. He also proved that arithmeticity is characterized in terms of the commensurator $\Comm(\Ga)$: a lattice  $\Ga$ in a semisimple Lie group is arithmetic if and only if $\Ga$ has infinite index in $\Comm(\Ga)$. The groups considered in this paper are lattices in solvable Lie groups, and there are many differences between the solvable and semisimple cases. For example, by work of Studenmund \cite[Thm.\ 1.2]{studenmund}, a lattice in a solvable Lie group  always has infinite index in its commensurator, independent of arithmeticity.
\end{rmk}

\begin{rmk}
A matrix $A\in\GL_n(\Z)$ induces a linear automorphism of $T^n\cong\R^n/\Z^n$. The associated mapping torus $E_A=\frac{T^n\ti[0,1]}{(x,1)\sim(Ax,0)}$ has fundamental group $\pi_1(E)\cong\Z^n\rt_A\Z$, and $E_A$ fibers as a $T^n$-bundle $E_A\ra S^1$ with monodromy $A$. Reducibility properties of $A$ translate to reducibility of the bundle $E_A\ra S^1$ in an obvious way. For example, by Theorem \ref{thm:real} if $\Z^n\rt_A\Z$ is arithmetic and $A$ has real eigenvalues, then $E_A$ has a finite cover $E\ra E_A$ whose induced bundle $E\ra S^1$ decomposes as a fiberwise product of $T^2$ bundles. In particular, arithmeticity puts a strong constraint on the topology of the bundle when the eigenvalues are real. This topological interpretation was one of the original motivations for this paper.  
\end{rmk}

\begin{rmk}\label{rmk:pseudo}Question \ref{q:arithmetic} is a variant of---and is motivated by---an open problem in the study of hyperbolic 3-manifolds, where one considers bundles $E_\phi\ra S^1$ with fiber a surface $\Sigma$ and pseudo-Anosov monodromy $\phi\in\pi_0\Homeo(\Si)$. In this setting, Thurston proved that $E_\phi$ admits a complete hyperbolic metric (unique by Mostow rigidity), and one can ask for a characterization of those $\phi$ for which $\pi_1(E_\phi)$ is arithmetic (in $\PSL_2(\C)$). This question seems to be wide open, except for a computer-assisted computation of  Bowditch--Maclachlan--Reid \cite{BMR} that gives a complete list of the arithmetic monodromies when $\Si=T^2\setminus\{\pt\}$ is a punctured torus. 
\end{rmk}

{\bf Techniques.} A central theme in the proofs of Theorems \ref{thm:real}--\ref{thm:prime-dim} is that various problems (such as arithmeticity of $\Ga_A$, irreducibility of $E_A$, or computing the rank of $S(\Z)$ for an algebraic torus $S$) can be translated into problems about algebraic tori and their character groups. The character group $X(T):=\Hom(T,G_m)$ of an algebraic torus $T$ carries an action of the Galois group of the splitting field of $T$, and this enables the use of Galois theory and representation theory to find examples with certain properties or prove that certain examples don't exist. Our proof of Theorem \ref{thm:examples} relies on the existence of number fields with Galois group $\Gal(P/\Q)$ isomorphic to the symmetric group and complex conjugation acting as a transposition. The existence of these number fields is ensured by known instances of the inverse Galois problem. The proof of Theorem \ref{thm:prime-dim} uses the classification of transitive permutation groups of prime degree and the representation theory of metacyclic groups. 

The main novelty of this paper is in the variety of techniques used to study Questions \ref{q:arithmetic} and \ref{q:irreducible}. These techniques, while well-known, connect algebraic groups, number theory, and group theory in a new way. 

{\bf Section outline.} Sections \ref{sec:groups}, \ref{sec:NT}, and \ref{sec:tori} contain background material: \S\ref{sec:groups} on the group theory of  $\Z^n\rt_A\Z$; \S\ref{sec:NT} on the conjugacy classification for hyperbolic, semisimple elements of $\GL_n(\Z)$; and \S\ref{sec:tori} on algebraic tori. Theorems \ref{thm:arithmeticity} and \ref{thm:real} are proved in \S\ref{sec:arithmeticity} and \S\ref{sec:real}, respectively. The final section \S\ref{sec:complex} contains proofs of Theorems \ref{thm:examples} and \ref{thm:prime-dim} as well as an example illustrating Theorem \ref{thm:examples}. 

{\bf Acknowledgements.} The author thanks B.\ Farb, from whom he learned about the question mentioned in Remark \ref{rmk:pseudo}, which ultimately led to this paper, and for valuable comments that improved the organization of this paper. He also thanks N.\ Salter for comments on a draft of this paper, and he thanks N.\ Salter and D.\ Studenmund for helpful conversations. 

\section{Group theory of $\Z^n\rt_A\Z$}\label{sec:groups}

In this section we collect some basic facts about the groups $\Ga_A=\Z^n\rt_A\Z$, their isomorphism classes, and their finite-index subgroups. 

\begin{lem}[Isomorphism classes]\label{lem:isomorphism}
Fix hyperbolic matrices $A,B\in\GL_n(\Z)$. Then $\Ga_A\cong\Ga_B$ if and only if $A$ is conjugate in $\GL_n(\Z)$ to one of $B,B^{-1}$. 
\end{lem}

\begin{proof}
We write elements of $\Ga_A$ as pairs $(x,i)\in\Z^n\ti\Z$ with multiplication 
\[(x,i)(y,j)=(x+A^iy,i+j).\] If $B=CAC^{-1}$, it is easy to check that $(x,i)\mapsto(Cx,i)$ defines an isomorphism $\Ga_A\ra\Ga_B$. If $B^{-1}=CAC^{-1}$, then we conclude that $\Ga_A\cong\Ga_B$ using the above argument together with the fact that $\Ga_B\cong\Ga_{B^{-1}}$ for any $B$. The latter isomorphism is easy to see from the point-of-view of mapping tori since the map $T^n\ti [0,1]\ra T^n\ti[0,1]$ defined by $(\theta,t)\mapsto(\theta,1-t)$ descends to a homeomorphism $E_B\cong E_{B^{-1}}$. 

For the converse, suppose that $\Phi:\Z^n\rt_A\Z\ra\Z^n\rt_B\Z$ is an isomorphism. First we show that $\Phi(\Z^n)=\Z^n$. Suppose that $(x,i)\in\Phi(\Z^n)$. Then also $(Bx,i)\in \Phi(\Z^n)$ because $\Phi(\Z^n)$ is normal in $\Ga_B$, and so $(Bx-x,0)=(Bx,i)(x,i)^{-1}$ is also in $\Phi(\Z^n)$. The vector $y:=Bx-x$ is nonzero because $B$ is hyperbolic. Now since $\Phi(\Z^n)$ is abelian, 
\[(0,0)=(x,i)(y,0)(x,i)^{-1}(y,0)^{-1}=(B^iy-y,0).\]This implies that $i=0$, again since $B$ is hyperbolic. Hence $\Phi$ restricts to $C:\Z^n\ra\Z^n$ for some $C\in\GL_n(\Z)$. 

Next write $\Phi(0,1)=(z,j)$. Computing $\Phi$ on $(0,1)(x,0)(0,-1)$ in two ways, we find that $B^jCx=CAx$ for all $x\in\Z^n$, which implies that $B^j=CAC^{-1}$. Here $j=\pm1$ because $\Phi(-C^{-1}z,1)=(0,j)$, 
which implies that $(-C^{-1}z,1)$ has a $j$-th root, so $j$ divides 1. 
\end{proof}

{\bf Commensurability.} Recall that groups $\Ga_1,\Ga_2$ are \emph{commensurable} if there is a group $\Ga_3$ that embeds as a finite-index subgroup $\Ga_3\hra\Ga_i$ for $i=1,2$. 

It is easy to show that any finite-index subgroup of $\Z^n\rt_A\Z$ has the form $L\rt_{A^k}\Z$, where $L\sbs\Z^n$ is an $A^k$-invariant sublattice. 

We say that $\Ga_{A_1}$ and $\Ga_{A_2}$ are \emph{fiberwise commensurable} if there exists $A_i$-invariant lattices $L_i\sbs\Z^n$ so that $L_1\rt_{A_1}\Z\cong L_2\rt_{A_2}\Z$. Using Lemma \ref{lem:isomorphism}, it is equivalent to say that the action $A_1\car L_1$ is isomorphic to either $A_2\car L_2$ or $A_2^{-1}\car L_2$. (We use the terminology \emph{fiberwise commensurable} because this definition is the group-theoretic version of the existence of a common fiberwise cover for the mapping tori $E_{A_1}$ and $E_{A_2}$.)

Fiberwise commensurability can be defined generally for semi-direct products, but it has a special property for the groups we're studying. 

\begin{lem}\label{lem:fiber-commensurable} Fix $A_1,A_2\in\GL_n(\Z)$ hyperbolic. If $\Ga_{A_1}$ and $\Ga_{A_2}$ are fiberwise commensurable, then $\Ga_{A_1}$ embeds as a finite-index subgroup of $\Ga_{A_2}$ (and vice versa). 
\end{lem}

\begin{proof}
First observe that if $L\sbs\Z^n$ is an $A$-invariant lattice, then $\Ga_A$ is a finite-index subgroup of $L\rt_A\Z$ (note that the other containment is obvious). To see this, choose $c\gg0$ so that $c\Z^n\sbs L\sbs\Z^n$. Then $c\Z^n$ is also $A$-invariant, and 
$\Ga_A=\Z^n\rt_A\Z\cong c\Z^n\rt_A\Z$ is a finite-index subgroup of $L\rt_A\Z$. 

Consequently, if $\Ga_{A_1},\Ga_{A_2}$ are fiberwise commensurable with $L_1,L_2\sbs\Z^n$ as in the definition, then one obtains an inclusion of finite-index subgroups
\[\Ga_{A_1}\hra L_1\rt_{A_1}\Z\cong L_2\rt_{A_2}\Z\hra\Ga_{A_2}.\qedhere\] 
\end{proof}

\section{Number-theoretic construction of integer matrices}\label{sec:NT}

In this section we recall the conjugacy classification of semisimple, hyperbolic elements of $\GL_n(\Z)$. This is needed for the proofs of Theorems \ref{thm:arithmeticity}, \ref{thm:real}, \ref{thm:prime-dim}. As a consequence of the classification, if $A\in\GL_n(\Z)$ has eigenvalues $\la_1,\ldots,\la_m$ with multiplicities $n_1,\ldots,n_m$, then one can construct a finitely-generated abelian group $M\sbs\bigoplus \Q(\la_i)^{n_i}$ that is invariant under the diagonal action of $\la=(\la_1,\ldots,\la_m)$ so that $\Z^n\rt_A\Z$ is isomorphic to $M\rt_\la\Z$ (in the simplest example, $M=\bigoplus\ca O_i^{n_i}$, where $\ca O_i\sbs\Q(\la_i)$ is the ring of integers).

The focus on this section is the following result that classifies hyperbolic, semisimple matrices $A$ whose characteristic polynomial $\chi(A)$ is fixed. 

\begin{thm}[Latimer--MacDuffee, Wallace, Husert]\label{thm:conjugacy}
Fix $d_1,n_1,\ldots,d_m,n_m\ge1$ and set $n=\sum_{i=1}^m d_in_i$. For each $i$, fix an algebraic unit $\la_i$ with minimal polynomial $\mu_i$ of degree $d_i$. Assume that no roots of $\mu_i$ lie on the unit circle and that $\mu_i\neq\mu_j$ for $i\neq j$. Then there is a bijection 
\[\left\{\begin{array}{c}\text{conjugacy classes of }\\\text{semisimple }A\in\GL_n(\Z) \\\text{ with }  \chi(A)=\mu_1^{n_1}\cd\mu_m^{n_m}\end{array}\right\}\longleftrightarrow
\left\{\begin{array}{c}
\text{module classes of finitely-generated }\\
\text{full $\Z[\la_1]\ti\cd\ti\Z[\la_m]$-modules }\\
M\sbs\Q(\la_1)^{n_1}\oplus\cd\oplus\Q(\la_m)^{n_m}\end{array}
\right\}.\]
\end{thm}

This theorem is due in various forms to Latimer--MacDuffee \cite{latimer-macduffee}, Wallace \cite{wallace}, and Husert \cite{husert}. See \cite[Thm.\ 1.4]{husert} for the statement given above. 

We explain how the bijection works, starting with two special cases. 
\begin{itemize}
\item When $m=1$ and $n_1=1$ the theorem classifies $A\in\GL_n(\Z)$ with a fixed irreducible characteristic polynomial $\mu$ (none of whose roots lie on the unit circle). If $\la$ is a root of $\mu$, then there is a bijection
\[\left\{\begin{array}{c}\text{conjugacy classes }\\\text{of } A\in\GL_n(\Z)\\ \text{ with } \chi(A)=\mu\end{array}\right\}\longleftrightarrow
\left\{\begin{array}{c}
\text{Ideal classes of }\\
\text{nonzero fractional ideals }\\I\sbs\Q(\la)\text{ of }\Z[\la] \end{array}
\right\}.\]
See \cite[Thm.\ 2]{wallace}. For the bijection, in one direction, given $A\in\GL_n(\Z)$, by basic linear algebra, one can find an eigenvector $Aw=\la w$ such that $w\in\Q(\la)^n$. If we write $w=(w_1,\ldots,w_n)$ in coordinates, then $I=\Z\{w_1,\ldots,w_n\}\sbs\Q(\la)$ is a fractional ideal of $\Z[\la]$ (the equation $Aw=\la w$ gives a way to rewrite $\la w_i$ as a $\Z$-linear combination of $w_1,\ldots,w_n$). For the other direction, given $I\sbs\Q(\la)$, choose a $\Z$-basis $w_1,\ldots,w_n\in I$. Then $\la w_1,\ldots,\la w_n\in I$ is also a $\Z$-basis for $I$. Take $A\in\GL_n(\Z)$ the matrix of the transformation $I\ra I$ taking $w_i$ to $\la w_i$ (with respect to the $w_i$-basis). See \cite{wallace} for more details. 

\item When $m=1$ and $n_1\ge1$, the theorem classifies semisimple $A\in\GL_n(\Z)$ with $\chi(A)=\mu^{n_1}$ with $\mu$ irreducible over $\Q$ (the degree of $\mu$ is $d_1$ and $n=d_1n_1$). If $\la$ is a root of $\mu$, then there is a bijection 
\[\left\{\begin{array}{c}\text{conjugacy classes of }\\\text{semisimple } A\in\GL_n(\Z)\\\text{ with } \chi(A)=\mu^{n_1}\end{array}\right\}\longleftrightarrow
\left\{\begin{array}{c}
\text{Module classes of}\\
\text{ finitely-generated, full }\\\text{ $\Z[\la]$-modules }M\sbs\Q(\la)^{n_1}\end{array}
\right\}.\]
A $\Z[\la]$-module $M\sbs\Q(\la)^{n_1}$ is called \emph{full} if its $\Q$-span is all of $\Q(\la)^{n_1}$. For the bijection, given $A\in\GL_n(\Z)$, choose linearly independent vectors $w^{(1)},\ldots,w^{(n_1)}\in\Q(\la)^{n}$ such that $Aw^{(j)}=\la w^{(j)}$ for each $1\le j\le n_1$. Next form an $(n\ti n_1)$-matrix whose $(i,j)$-entry is $w_i^{(j)}$, the $i$-th coordinate of $w^{(j)}$. Then the rows $w_i=(w^{(1)}_i,\ldots,w^{(n_1)}_i)\in\Q(\la)^{n_1}$ generate a full $\Z[\la]$-module $M=\Z\{w_1,\ldots,w_n\}$. Conversely, given a full $\Z[\la]$-module $M\sbs\Q(\la)^{n_1}$, choose a basis $M=\Z\{w_1,\ldots,w_n\}$, and take $A$ to be the matrix of multiplication by $\la$ on $M$ with respect to the given basis. See \cite[\S1.2]{husert} for more details. 
\end{itemize} 

For the general case of Theorem \ref{thm:conjugacy}, the bijection works similarly. Given semisimple $A\in\GL_n(\Z)$ with $\chi(A)=\mu_1^{n_1}\cd\mu_m^{n_m}$, conjugate $A$ to a block diagonal matrix $(A_1,\cdots, A_m)$, where $\chi(A_i)=\mu_i^{n_i}$. Then repeat the construction of the preceding paragraph for each $A_i$ to get a full $\Z[\la_i]$-module $M_i\sbs\Q(\la_i)^{n_i}$. Then set $M=M_1\oplus\cdots\oplus M_m$. The construction in the reverse direction is also similar to what was discussed above. 

In summary, given semisimple, hyperbolic $A\in\GL_n(\Z)$ with $\chi(A)=\mu_1^{n_1}\cdots\mu_m^{n_m}$, there exists a full submodule $M\sbs\Q(\la_1)^{n_1}\oplus\cd\oplus\Q(\la_m)^{n_m}$, where $\la_i$ is a root of $\mu_i$, so that the action of $A$ on $\Z^n$ is isomorphic (as $\Z[\Z]$-modules) to the diagonal action of $(\la_1,\ldots,\la_m)$ on $M\cong\Z^n$. 

\section{Proof of Theorem \ref{thm:arithmeticity}}\label{sec:arithmeticity}

Theorem \ref{thm:arithmeticity} is proved in two steps that are carried out in \S\ref{sec:char-poly} and \S\ref{sec:grunewald-platonov}. 

\un{Step 1}. We show that if $A,B\in\GL_n(\Z)$ are semisimple with the same characteristic polynomial, then $\Ga_A$ is arithmetic if and only if $\Ga_B$ is arithmetic. This allows us to reduce the proof of Theorem \ref{thm:arithmeticity} to the case $\Ga_\la=\bigoplus\ca O_i^{n_i}\rt_\la\Z$, where $\la$ and $n_i$ are as in the statement of Theorem \ref{thm:arithmeticity} and $\ca O_i\sbs\Q(\la_i)$ is the ring of integers. 

\un{Step 2}. We solve the arithmeticity problem for $\Ga_\la=\bigoplus \ca O_i^{n_i}\rt_\la\Z$. The case when $m=1$ and $n_1=1$, i.e.\ $\Ga=\ca O_K\rt_\la\Z$ for some $\la\in\ca O_K^\ti$, was solved by Grunewald--Platonov \cite{grunewald-platonov-polycyclic}, and we adapt their argument to the general case.

\subsection{Characteristic polynomial and fiberwise commensurability}\label{sec:char-poly}

Fix $A\in\GL_n(\Z)$ with $\la=(\la_1,\ldots,\la_m)$ and $n_1,\ldots,n_m$ as in the statement of Theorem \ref{thm:arithmeticity}. By \S\ref{sec:NT}, the action of $A$ on $\Z^n$ has the same characteristic polynomial as the action of $\la=(\la_1,\ldots,\la_m)$ on $M=\bigoplus\ca O_i^{n_i}$, where $\ca O_i\sbs\Q(\la_i)$ is the ring of integers. By the following proposition, $\Z^n\rt_A\Z$ is arithmetic if and only if $\Ga_\la=M\rt_\la\Z$ is arithmetic. 

\begin{lem}[Arithmeticity depends only on characteristic polynomial]\label{lem:char-poly}
Fix semisimple, hyperbolic $A,B\in\GL_n(\Z)$. If $\chi(A)=\chi(B)$, then 
\begin{enumerate}[(i)]
\item the groups $\Ga_A$ and $\Ga_B$ are fiberwise commensurable, and 
\item the group $\Ga_A$ is arithmetic if and only if $\Ga_B$ is arithmetic. 
\end{enumerate} 
\end{lem}

\begin{proof}
First we note that (i) implies (ii): by Lemma \ref{lem:fiber-commensurable}, if $\Ga_A$ and $\Ga_B$ are fiberwise commensurable, then $\Ga_A$ is a finite-index subgroup of $\Ga_B$, and vice versa. This implies (ii) since arithmeticity is obviously inherited by finite-index subgroups. 

Proof of (i): To show $\Ga_A$ and $\Ga_B$ are fiberwise commensurable, write $\chi=\mu_1^{n_1}\cd\mu_m^{n_m}$ for the common characteristic polynomial. Choose a root $\la_i$ of $\mu_i$ for each $i$. From Theorem \ref{thm:conjugacy} and the discussion in \S\ref{sec:NT}, there exist $\prod\Z[\la_i]$-modules $M_A,M_B\sbs\bigoplus\Q(\la_i)^{n_i}$ so that $M_A\rt_\la\Z\cong\Ga_A$, and similarly for $B$. The intersection $M_A\cap M_B$ is also a full $\prod\Z[\la_i]$-module, so $(M_A\cap M_B)\rt_\la\Z$ is a finite index subgroup of both of $M_A\rt_\la\Z$ and $M_B\rt_\la\Z$. This shows $\Ga_A,\Ga_B$ are fiberwise commensurable, as desired. 
\end{proof}


\begin{rmk}\label{rmk:converse}
We note for later use that the following converse of Proposition \ref{lem:char-poly} is also true: if $\Ga_{A}$ and $\Ga_{B}$ are fiberwise commensurable, then $\chi(A)=\chi(B)$. This follows quickly from the fact that if $L\sbs\Z^n$ is an $A$-invariant lattice, then the linear maps   $A:\Z^n\ra\Z^n$ and $A: L\ra L$ have the same characteristic polynomial. \end{rmk}

\subsection{Quasi-split tori and arithmeticity}\label{sec:grunewald-platonov}

We prove Theorem \ref{thm:arithmeticity} for $\Ga_\la=\bigoplus\ca O_i^{n_i}\rt_\la\Z$, where $\la=(\la_1,\ldots,\la_m)$ and $\ca O_i\sbs\Q(\la_i)$ is the ring of integers. For notational simplicity, we set $M=\bigoplus\ca O_i^{n_i}$ and $K_i=\Q(\la_i)$. 

Consider the algebraic group 
\[T=\prod R_{K_i/\Q}(G_m),\]
where $G_m$ is the multiplicative group, and $R_{K_i/\Q}(\cdot)$ is the restriction of scalars functor (see e.g.\ \cite[\S2.1.2]{platonov-rapinchuk}). Then $T$ is defined over $\Q$, and 
\[T(\Q)=K_1^\ti\ti\cd\ti K_m^\ti\>\>\>\text{ and }\>\>\>T(\Z)=\ca O_1^\ti\ti\cd\ti\ca O_m^\ti.\]
In particular, $\la\in T(\Z)$. Let $S\sbs T$ be the Zariski closure of the subgroup $\pair{\la}\sbs T$. In Proposition \ref{prop:arithmeticity} below we prove that $\Ga_\la$ is arithmetic if and only if $S(\Z)$ has rank 1. This will finish the proof of Theorem \ref{thm:arithmeticity}. 

Terminology: If $\Lambda\sbs S(\Q)$ is commensurable with $S(\Z)$, we say $\Lam$ is an \emph{arithmetic subgroup of $S$}.

\begin{prop}[Arithmeticity criterion]\label{prop:arithmeticity}
Fix $\Ga_\la=M\rt_\la\Z$ and $T$ as above. Set $\Lam=\pair{\la}\sbs T(\Z)$, and let $S\sbs T$ be the Zariski closure of $\Lambda$. Then $\Ga_\la$ is arithmetic if and only if $\Lam$ is an arithmetic subgroup of $S$. 
\end{prop}

In the situation of Proposition \ref{prop:arithmeticity}, $\Lam\cong\Z$ is a subgroup of $S(\Z)$, which is abelian, so $\Lam\sbs S$ is an arithmetic subgroup if and only if $\rank S(\Z)=1$. Thus the conclusion of Proposition \ref{prop:arithmeticity} gives us the desired conclusion for Theorem \ref{thm:arithmeticity}.

In the case $m=1$ and $n_1=1$, i.e.\ $\Ga_\la=\ca O_K\rt_\la\Z$ with $\la\in\ca O_K^\ti$, the proof of Proposition \ref{prop:arithmeticity} is given in \cite[Prop.\ 3.1]{grunewald-platonov-polycyclic}. The argument in the general case is a straightforward generalization, as we explain next. At times we refer to \cite{grunewald-platonov-polycyclic} for further details.

\begin{proof}[Proof of Proposition \ref{prop:arithmeticity}]
The ``if" statement is easy. If $\Lam\sbs S$ is an arithmetic subgroup, then $\Ga_\la$ is an arithmetic subgroup in $\prod R_{K_i/\Q}(G_a)^{n_i}\rt S$, where $G_a$ denotes the additive group. 

Now we prove the ``only if" statement. Assuming that $\Ga_\la$ is arithmetic, there exists a solvable algebraic group $H$ defined over $\Q$ and an embedding $j:\Ga_\la\hra H(\Q)$ whose image $\Ga=j(\Ga_\la)$ is commensurable to $H(\Z)$. 

Note: in general a solvable group $\Ga$ may be realized in as a lattice in different $H$ (this is already true for $\Ga=\Z$). The proof identifies $H$ by determining its unipotent radical and maximal torus. 

{\it Claim $1$.} The group $j(\Lam)$ is an arithmetic subgroup of its Zariski closure $S_0\sbs H$.

Here we are identifying $\Lam=\pair{\la}$ with the obvious $\Z$ subgroup of $\Ga_\la=M\rt_\la\Z$. 

{\it Proof of Claim $1$.} First observe that $\Ga\cap S_0$ is an arithmetic subgroup of $S_0$ (this follows easily from the definition of arithmeticity), so it suffices to show that $j(\Lam)=\Ga\cap S_0$.  Note that $S_0$ is abelian. In addition $\Lam\sbs\Ga_\la$ is a maximal abelian subgroup because no coordinate of $\la=(\la_1,\ldots,\la_m$) is a root of unity. Since $\Ga\cap S_0$ is an abelian subgroup of $\Ga$ and contains $j(\Lam)$, maximality of $j(\Lam)$ implies that $\Ga\cap S_0=j(\Lam)$, proving Claim 1. 

Next we relate $S_0$ and $T=\prod R_{K_i/\Q}(G_m)$. Here we use the unipotent radical $U(H)\sbs H$. As before $\Ga\cap U(H)$ is an  arithmetic subgroup of $U(H)$. 

{\it Claim $2$.} The group $\Ga\cap U(H)$ is equal to $j(M)$. 

Before explaining the claim, recall that the \emph{fitting subgroup} of a polycyclic group $\Ga$ is the unique maximal normal nilpotent subgroup (for $\Ga_\la=M\rt_\la\Z$, the fitting subgroup is $M$). We also use the notion of a \emph{reduced} solvable algebraic group over $\Q$; see \cite[\S2]{grunewald-platonov-polycyclic} for the definition. 

{\it Proof of Claim $2$.} This is true when $H$ is reduced since this implies that $\Ga\cap U(H)$ is the fitting subgroup of $\Ga$ (see \cite[Lem.\ 2.1]{grunewald-platonov-polycyclic}), and the fitting subgroup of $\Ga$ is equal to $j(M)$. Assuming that $H$ is reduced does not result in any loss of generality by \cite[Thm.\ 2.2]{grunewald-platonov-polycyclic}. This proves Claim 2.

Next we identify $U(H)$, using that $j(M)=\Ga\cap U(H)$ is an arithmetic subgroup. Choosing a $\Z$-basis for $M=\bigoplus\ca O_i^{n_i}$ gives a $\Q$-basis for $M_\Q:=\bigoplus K_i^{n_i}$ and also for the Lie algebra $\mf u(H)(\Q)$, yielding an isomorphism $M_\Q\cong U(H)(\Q)$ as in the diagram below.
\[\begin{xy}
(-15,15)*+{M}="A";
(15,15)*+{U(H)(\Q)}="B";
(-15,-0)*+{M_\Q}="C";
(15,-0)*+{\mf u(H)(\Q)}="D";
{\ar "A";"B"}?*!/_3mm/{j};
{\ar "B";"D"}?*!/_6mm/{\exp^{-1}};
{\ar@{^{(}->} "A";"C"}?*!/^6mm/{};
{\ar@{-->} "C";"D"}?*!/_3mm/{\Theta};
\end{xy}\]
The exponential map here is an isomorphism; see \cite[\S2]{grunewald-platonov-polycyclic} for more detail. Let $\be:T\ra\GL_n$ be the morphism induced by the action of $\prod K_i^\ti$ on $M_\Q\cong\Q^n$ (using the chosen basis for $M$). Via $\Theta$, we can identify this with the adjoint action of $j(\Lam)$ on $\mf u(H)$:
\[\begin{xy}
(-15,15)*+{\Lam}="A";
(15,15)*+{j(\Lam)\sbs S_0}="B";
(-15,-0)*+{T}="C";
(15,-0)*+{\GL_n}="D";
{\ar "A";"B"}?*!/_3mm/{j};
{\ar "B";"D"}?*!/_8mm/{\text{adjoint}};
{\ar "A";"C"}?*!/^6mm/{};
{\ar "C";"D"}?*!/_3mm/{\be};
\end{xy}\]
Define $S=\text{adjoint}(S_0)$. Then $\text{adjoint}\big(j(\Lam)\big)\sbs S$ is Zariski dense because $j(\Lam)\sbs S_0$ is Zariski dense. Also $\text{adjoint}(j(\Lam))\sbs S$ is an arithmetic subgroup because $j(\Lam)\sbs S_0$ is \cite[Thm.\ 6]{borel-density}. Since $\be$ (as multiplication of $\prod K_i^\ti$ on $\bigoplus K_i^{n_i}$) is injective and the diagram commutes, we identify $S$ with a subgroup of $T$ so that $\Lam\sbs S$ is an arithmetic subgroup. This completes the proof of the proposition and the proof of Theorem \ref{thm:arithmeticity}. 
\end{proof}

\begin{rmk}\label{rmk:supergroup}
When $\la\in T(\Z)$ has infinite order, then for $k\ge2$, the subgroups $\pair{\la^k}$ and $\pair{\la}$ have the same Zariski closure in $T$. To see this, let $S$ be the Zariski closure of $\pair{\la}$. Multiplication by $k$ defines a surjective morphism $S\ra S$ that sends $\pair{\la}$ to $\pair{\la^k}$. Then if $\pair{\la}\sbs S$ is an arithmetic subgroup, this implies $\pair{\la^k}$ is also an arithmetic subgroup of $S$ by \cite[Thm.\ 6]{borel-density}. Hence $\pair{\la^k}$ is Zariski dense in $S$. Consequently, Proposition \ref{prop:arithmeticity} implies that if $M\rt_{\la^k}\Z$ is arithmetic, then also $M\rt_\la\Z$ is arithmetic. 
\end{rmk}

\section{Algebraic tori, character groups, and integral points}\label{sec:tori}

In this section we recall some facts about algebraic tori that will be needed in Sections \ref{sec:real} and \ref{sec:complex}. Our main reference is \cite{platonov-rapinchuk}. 

{\bf Algebraic tori and their character groups.} An \emph{algebraic torus} over $\Q$ is an algebraic group that is isomorphic to $(G_m)^r$ over $\bar\Q$, where $r=\dim T$. The \emph{character group} $X(T):=\Hom(T, G_m)$ of $T$ is isomorphic to $\Z^r$ and has the structure of a $\Z[\bb G]$-module, where $\bb G=\Gal(\bar\Q/\Q)$ is the absolute Galois group. The functor $T\mapsto X(T)$ defines a contravariant equivalence of categories between (algebraic tori defined over $\Q$) and (finitely-generated free-abelian $\Z[\bb G]$-modules) \cite[Thm.\ 2.1]{platonov-rapinchuk}. For example, a surjection $X(T)\onto X(S)$ between  character groups is induced by an embedding $S\hra T$, and conversely.

The action of $\bb G$ on $X(T)$ defines a homomorphism $\rho:\bb G\ra\GL_r(\Z)$. The fixed field of $\ker(\rho)\sbs\bb G$ is a finite Galois extension $P/\Q$, which is the \emph{splitting field} of $T$. It is the smallest field that satisfies the following equivalent properties (c.f.\ \cite[\S2.1.7]{platonov-rapinchuk}): 
\begin{enumerate}[(i)]
\item The torus $T$ is \emph{$K$-split}, i.e.\ there is an isomorphism $T\cong (G_m)^r$ defined over $K$. 
\item Every character $T\ra G_m$ is defined over $K$. 
\end{enumerate} 
The image $\im(\rho)\cong\bb G/\ker(\rho)$ is isomorphic to $\Gal(P/\Q)$. In particular, if $P$ is the splitting field for $T$, then $\Gal(P/\Q)$ acts faithfully on $X(T)$. 

Recall that a map of algebraic groups $T_1\ra T_2$ is an \emph{isogeny} if it is surjective with finite kernel. If $T_1,T_2$ are tori, there is an isogeny between them if and only if $X(T_1)\ot\Q\cong X(T_2)\ot\Q$ as $\Q[\bb G]$-modules. In particular, isogeny of algebraic tori is an equivalence relation \cite[\S2.1.7]{platonov-rapinchuk}. 

{\bf Quasi-split tori.} A torus is called \emph{quasi-split} if it is a finite product of tori of the form $R_{K/\Q}(G_m)$, c.f.\ \S\ref{sec:grunewald-platonov}. The field norm $N_{K/\Q}:K\ra\Q$ defines a character $N:R_{K/\Q}(G_m)\ra G_m$ that's defined over $\Q$, and its kernel is called the \emph{norm torus}, denoted  $R_{K/\Q}^1(G_m)$. The norm torus has $\Q$-rank 0, and its set of $\Z$-points is equal to the units in $K$ of norm 1, denoted $\ca O_K^1$. 

Quasi-split tori are characterized by the property that their character groups are permutations modules. If $T=R_{K/\Q}(G_m)$ and $P$ is the Galois closure of $K/\Q$, then $X(T)\cong\Z[G/H]$, where $G=\Gal(P/\Q)$ and $H=\Gal(P/K)$. Furthermore, the character group $X(T^1)$ of the norm torus $T^1\sbs T$  is isomorphic to the quotient $\Z[G/H]/\Z$ by the trivial sub-representation \cite[\S2.1.7]{platonov-rapinchuk}. 

A decomposition of $\Q[G/H]$ as a $G$-representation leads to a decomposition of $T=R_{K/\Q}(G_m)$, up to isogeny. More precisely, choose a decomposition $\Q[G/H]\cong\bigoplus X_{i,\Q}$ into irreducible representations of $G$, and define $X_i$ as the image of $\Z[G/H]$ under the projection $\Q[G/H]\onto X_{i,\Q}$. This leads to a commutative diagram 
\[\begin{xy}
(-20,15)*+{X(T)}="A";
(10,15)*+{\bigoplus X_i}="B";
(-20,-0)*+{X(T)\ot\Q}="C";
(10,-0)*+{\bigoplus X_{i,\Q}}="D";
{\ar "A";"B"}?*!/_3mm/{\phi};
{\ar@{^{(}->} "B";"D"}?*!/_3mm/{};
{\ar@{^{(}->} "A";"C"}?*!/^3mm/{};
{\ar "C";"D"}?*!/_3mm/{\cong};
\end{xy}\]
From the diagram, one finds that $\phi$ is injective and has finite cokernel. It follows that $T$ is isogenous to $\prod T_i$, where $T_i$ is the torus with $X(T_i)\cong X_i$. 

{\bf Integer points of a torus.} In order to apply the arithmeticity criterion Proposition \ref{prop:arithmeticity}, we need to be able to determine the rank of $S(\Z)$ as an abelian group for an arbitrary torus $S$. This rank is given by the following formula, which generalizes Dirichlet's unit theorem \cite[\S4.5]{platonov-rapinchuk}
\begin{equation}\label{eqn:rank}\rank S(\Z)=\rank_\R(S)-\rank_\Q(S).\end{equation}
Recall that the $\Q$-rank is the dimension of the largest subtorus $\Q$-split torus, and similarly for $\R$-rank. Next we explain how to compute the $\R$- and $\Q$-ranks of $S$ in terms of $X(S)$.  First observe that if $L$ is the splitting field of $S$, then $\rank_\Q(S)=\rank X(S)^{\Gal(L/\Q)}$ since a character is defined over $\Q$ if and only if it is fixed by the action of $\Gal(L/\Q)$. 

\begin{lem}\label{lem:rank}
Fix an algebraic torus $S$ defined over $\Q$, and denote its splitting field by $L$. Let $\tau\in\Gal(L/\Q)$ be complex conjugation. Then the $\R$-rank of $S$ is equal to the rank of $X(S)^\tau$ as an abelian group. 
\end{lem}

\begin{proof}
First we identify $S$ as a torus over $\R$. Consider the action of $\Gal(\C/\R)=\pair{\tau}$ on $X(S)$. Any free-abelian $\Z[\Z/2\Z]$-module decomposes as a direct sum of copies of the trivial representation $\Z$, the sign representation $\Z_-$, and the group ring $\Z[\Z/2\Z]$ (see e.g.\ \cite[\S74]{curtis-reiner-RTFGAA}). Hence, there is a decomposition
\[X(S)\cong\Z[\Gal(\C/\R)]^\al\oplus\Z^\be\oplus\Z_-^\ga,\]
for some $\al,\be,\ga\ge0$. The rank of $X(S)^\tau$ is equal to $\alpha+\beta$. Using the correspondence between $\Z[\Gal(\C/\R)]$-modules and tori defined over $\R$, we conclude that there is an isomorphism 
\[S\cong  R_{\C/\R}(G_m)^\al\ti G_m^\be\ti R_{\C/\R}^1(G_m)^\ga\]
defined over $\R$. (Compare with \cite[\S2.2.4]{platonov-rapinchuk}.) The $\R$-ranks of $R_{\C/\R}(G_m)$ and $R_{\C/\R}^1(G_m)$ are $1$ and $0$, respectively. Thus $\rank_\R(S)=\alpha+\beta$. This is equal to $\rank X(S)^\tau$, as computed above, which proves the lemma. 
\end{proof}

\section{Proof of Theorem \ref{thm:real}} \label{sec:real}

We begin with the easier direction. 

\un{(ii) implies (i)}. 
By assumption, after replacing $A$ with a power, we have $\chi(A)=\mu_1^{n_1}\cdots\mu_m^{n_m}$, where $\mu_i$ is an degree-2 polynomial. Denoting $\la_i$ a root of $\mu_i$, we are assuming that there exists a quadratic extension $L/\Q$ and a unit $\ep\in\ca O_L^\ti$ so that $\la_i=\ep^{\ell_i}$ for some $\ell_i\neq0$. 

We want to show that $\Ga_A$ is arithmetic. Note that our replacement of $A$ with $A^k$ at the beginning does not change arithmeticity of $\Ga_A$ by Remark \ref{rmk:supergroup}. By \S\ref{sec:NT} and Proposition \ref{lem:char-poly}, we can replace $\Ga_A$ with the fiberwise commensurable subgroup
\begin{equation}\label{eqn:torus-group}
(\ca O_L^{n_1}\oplus\cdots\oplus\ca O_L^{n_m})\rt_{(\ep^{\ell_1},\ldots,\ep^{\ell_m})}\Z.\end{equation}
We show this group is arithmetic using Proposition \ref{prop:arithmeticity}. Setting $S=R_{L/\Q}(G_m)$ and $T=S^m$, the Zariski closure of $\pair{(\ep^{\ell_1},\ldots,\ep^{\ell_m})}\sbs T$ is the subgroup 
\[\Delta=\{x\in S^m: \phi(x_i)^{\ell_j}=\phi(x_j)^{\ell_i}\text{ for all }\phi\in X(S)\}.\] 
Note that $\Delta\cong S$, so $\rank\Delta(\Z)=\rank S(\Z)=1$. Proposition \ref{prop:arithmeticity} implies that the group in (\ref{eqn:torus-group}) is arithmetic, and so $\Ga_A$ is also arithmetic.

\un{(i) implies (ii)}. Fix $A\in\GL_n(\Z)$ such that $\Ga_A$ is arithmetic, and assume all the eigenvalues of $A$ are real. We want to show that there is $k\ge0$ and a quadratic extension $L/\Q$ so that the eigenvalues of $A^k$ are powers of a fundamental unit $\ep\in\ca O_L^\ti$. 

We use arithmeticity of $\Ga_A$ to obtain information about the eigenvalues of $A$. Write $\chi(A)=\mu_1^{n_1}\cd\mu_m^{n_m}$ and $\la=(\la_1,\ldots,\la_m)$ with $\la_i$ a root of $\mu_i$. Denote the ring of integers $\ca O_i\sbs \Q(\la_i)$ and set $M=\bigoplus\ca O_i^{n_i}$. By \S\ref{sec:NT} and Proposition \ref{lem:char-poly}, arithmeticity of $\Ga_A$ implies arithmeticity of $\Ga_\la=M\rt_\la\Z$.

Set $T_i=R_{\Q(\la_i)/\Q}(G_m)$ and $T=\prod T_i$. Write $\la=(\la_1,\ldots,\la_m)\in\prod T_i(\Z)=T(\Z)$. Observe that $\la^2\in T^1(\Z)$, where $T^1=\prod R_{K_i/\Q}^1(G_m)\sbs T$. To see this, consider the short exact sequence (c.f.\ \S\ref{sec:tori})
\[0\ra R_{K_i/\Q}^1(G_m)\hra R_{K_i/\Q}(G_m)\xra{N}G_m\ra0.\]
Restricting attention to the $\Z$-points, since $G_m(\Z)=\{\pm1\}$, it follows that $\la_i^2\in R_{K_i/\Q}^1(G_m)$. We replace $\la$ by $\la^2$ (which corresponds to replacing $A$ by $A^2$).

Let $S\sbs T^1$ be the Zariski closure of $\pair{\la}$. Observe that $\dim(S)=1$ because 
\[1=\rank S(\Z)=\rank_\R(S)-\rank_\Q(S)=\rank_\R(S)=\dim(S)\]
The first equality holds because $\Ga_\la$ is arithmetic, c.f.\ Proposition \ref{prop:arithmeticity}; the second equality is Equation (\ref{eqn:rank}); the third equality holds because $S\sbs T^1$ and $\rank_\Q(T^1)=0$ (see \S\ref{sec:tori}); the final equality holds because $T$ and hence $S$ is defined over $\R$, by the assumption on eigenvalues of $A$. 

By definition $\dim(S)=1$ means that $X(S)\cong\Z$. Since the splitting field $L/\Q$ of $S$ embeds in $\Aut(X(S))\cong\{\pm1\}$, the group $\Gal(L/\Q)$ is either trivial or $\Z/2\Z$. If $\Gal(L/\Q)$ were trivial, then $S$ would split over $\Q$, so we conclude $\Gal(L/\Q)=\Z/2\Z$. Thus $L/\Q$ is a real quadratic extension. 

Next we show that the embedding $S\hra T^1$ factors through a diagonal embedding
\[S\hra S\ti \cdots\ti S\hra T_1^1\ti\cdots\ti T_m^1.\]
This is achieved by studying the surjection $f:X(T^1)\ra X(S)$ induced by the inclusion $S\hra T^1$. For each $i$, let $P_i$ be the Galois closure of $K_i/\Q$. The splitting field $P$ of $T$ is the smallest Galois extension of $\Q$ containing all the $P_i$. Denote $G=\Gal(P/\Q)$ and $H_i=\Gal(P/K_i)$. Then $X(T_i)\cong\Z[G/H_i]$ and $X(T^1)\cong\Z[G/H_i]/\Z$ (as discussed in \S\ref{sec:tori}). Denote $G':=\Gal(L/\Q)$. Since $L\sbs P$, there is a surjection $G\ra G'$ given by restricting an automorphism of $P$ to $L$. 

The map $f:X(T^1)\ra X(S)$ is a $\Z[G]$-module map. For each $i$, it restricts to the summand $X(T_i^1)\sbs X(T^1)$ giving a map
\[f_i:X(T_i^1)\cong\Z[G/H_i]/\Z\ra\Z[G']/\Z\cong X(S).\]
Such a map is determined by $f([eH_i])$. There are two cases to consider, depending on whether or not $H_i$ is a subgroup of $\ker(G\ra G')$. If $H_i\not\sbs\ker(G\ra G')$, then the only $\Z[G]$-module map $\Z[G/H_i]/\Z\ra\Z[G']/\Z$ is the zero map. If $H_i\sbs\ker(G\ra G')$, then for each $n\in\Z\cong\Z[G']/\Z$, there is a unique equivariant map with $f([eH_i])=n$. 

Observe that the map $f_i$ is nonzero for each $i$. To see this, suppose that $f_i=0$ for some $i$. Then $S\sbs T_1\ti\cd\ti T_m$ is contained in 
\[T_1\ti\cd\ti T_{i-1}\ti0\ti T_{i+1}\ti\cd\ti T_m,\]
but this forces $\la_i=0$, which leads to a contradiction. Then each $f_i$ is nontrivial, and surjects onto $\ell_i X(S)$ for some $\ell_i\in\Z\setminus\{0\}$. Note that $\ell_i X(S)$ is isomorphic to $X(S)$ as a $\Z[G]$ module, so the surjection $X(T_i^1)\onto \ell_i X(S)$ corresponds to an embedding $S\hra T_i^1$. 

Thus, our surjective map $f:X(T^1)\onto X(S)$ factors through surjections 
\[X(T_1^1)\oplus\cdots\oplus X(T_m^1)\onto X(S)\oplus\cdots\oplus X(S)\onto X(S).\]
The first map is given by $(\phi_1,\ldots,\phi_m)\mapsto (f_1(\phi_1),\ldots,f_m(\phi_m))$ and the second map is given by $(\psi_1,\ldots,\psi_m)\mapsto \psi_1+\cdots+\psi_m$. The latter surjection gives an embedding $S\hra S\ti\cd\ti S$. Observe that $S(\Z)\cong\Z\ti\Z/2\Z$ since $S=R_{L/\Q}^1(G_m)$ with $L/\Q$ a real quadratic extension (the torsion subgroup of $\ca O_L^\ti$ consists of roots of unity, so is $\{\pm1\}$ in the totally real case). Choose a fundamental unit $\ep\in S(\Z)$. After replacing $\la$ by $\la^2$, we can assume that $\la_i\in S(\Z)$ is a power of $\ep$ for each $i$. 

In summary, we've shown that if we replace $A$ with $A^2$, then the eigenvalues are all units in a quadratic extension $L/\Q$, and if we replace $A$ with $A^4$, then the eigenvalues are all powers of the fundamental unit $\ep\in\ca O_L^\ti$. This finishes the proof of ``(i) implies (ii)" and completes the proof of Theorem \ref{thm:real}.\qed

\section{Monodromies with complex eigenvalues}\label{sec:complex}

In \S\ref{sec:real} we related the arithmeticity of $\Ga_A=\Z^n\rt_A\Z$ to the reducibility of $\chi(A)$ when the eigenvalues of $A\in\GL_n(\Z)$ are real. Now we remove this restriction on eigenvalues, and focus on Question \ref{q:irreducible}. We prove Theorem \ref{thm:examples} in \S\ref{sec:examples} and Theorem \ref{thm:prime-dim} in \S\ref{sec:prime-dim}. 

\subsection{Proof of Theorem \ref{thm:examples}}\label{sec:examples}

To prove Theorem \ref{thm:examples}, given $k\ge1$, we find $n\ge k$, a degree-$n$ extension $K/\Q$, and $\la\in\ca O_K^\ti$ so that (1) the action of $\la$ on $\ca O_K\cong\Z^n$ is fully irreducible, and (2) the group $\Ga_\la=\ca O_K\rt_\la\Z$ is arithmetic. Our first step is to give a sufficient condition for the action of $\la$ on $K$ (and hence $\ca O_K$) to be (fully) irreducible. 

\begin{lem}[Irreducibility test]\label{lem:irreducible}
Let $K$ be a number field, and fix $\la\in\ca O_K^\ti$. Denote $P$ the Galois closure of $K/\Q$. Set $T=R_{K/\Q}(G_m)$, and let $S\sbs T$ be the Zariski closure of $\pair{\lambda}$. Then 
\begin{enumerate}[(a)]
\item The action of $\la$ on $K$ is irreducible if and only if $K=\Q(\la)$.  
\item If the action of $\la^k$ on $K$ is reducible for some $k\ge1$, then there exists a $\Z[\Gal(P/\Q)]$ permutation module $\Z[Y]$ so that the surjection $X(T)\onto X(S)$, induced by the inclusion $S\hra T$, factors through surjective maps \[X(T)\onto\Z[Y]\onto X(S)\]
of $\Z[\Gal(P/\Q)]$-modules. 
\end{enumerate} 
\end{lem}

\begin{proof}[Proof of Lemma \ref{lem:irreducible}]
(a) If $K=\Q(\la)$, then since $\Q(\la)\cong\Q[t]/(\mu)$, where $\mu$ is the minimal polynomial of $\la$, then $\mu$ (which is the characteristic polynomial of the action of $\la$ on $K$) is irreducible because $K$ is a field. Conversely, if $L:=\Q(\la)$ is properly contained in $K$, then by the primitive element theorem, we can write $K=L(\ep)$ for some $\ep\in K$. Then $\la$ preserves the decomposition $K\cong L\oplus \ep L\oplus\cdots\oplus \ep^dL$, showing the action is reducible. 

(b) Using part(a), if $\la^k$ acts reducibly, take $L$ with $\la^k\in L\subsetneq K$. The inclusions $S\hra T$ factors through $T'=R_{L/\Q}(G_m)$, so there are surjections $X(T)\onto X(T')\onto X(S)$. Since $T'$ is quasi-split, $X(T')$ is a permutation module, as desired. (Here we have used that $\pair{\la}$ and $\pair{\la^k}$ have the same Zariski closure, c.f.\ Remark \ref{rmk:supergroup}.)
\end{proof}

Next we find algebraic tori $S$ with $\dim(S)$ large and $\rank S(\Z)=1$. 

\begin{prop}[High-dimensional tori with small $\R$-rank]\label{prop:inverse-galois}
For every $k\ge1$, there exists an algebraic torus $S$ defined over $\Q$ so that $\rank_\Q(S)=0$, $\rank_\R(S)=1$, and $\dim(S)=k$. 
\end{prop}

\begin{proof}
This is easy for $k=1,2$. For $k=1$ one may take $S=R_{L/\Q}^1(G_m)$, where $L/\Q$ is a real quadratic extension. For $k=2$, consider $S=R_{L/\Q}^1(G_m)$, where $L/\Q$ is an imaginary cubic extension. 

Assume $k\ge3$ (the following construction will also work for $k=2$). First we build a $\Z[S_{k+1}]$-module $X$ so that a transposition $\tau\in S_{k+1}$ acts in a special way. Then we use the inverse Galois problem to show there is a torus $S$ with $X(S)=X$. 

Consider the $\Q[S_{k+1}]$ module $X_\Q=\Q^{k}\ot\Q_-$, where $\Q^k$ is the standard representation and $\Q_-$ is the sign representation. Choose a module map $\Q[S_{k+1}]\onto X_\Q$, and let $X\sbs X_\Q$ be the image of $\Z[S_{k+1}]\sbs\Q[S_{k+1}]$. The module $X$ has a finite-index submodule isomorphic to $\Z^k\ot\Z_-$. Then for any transposition $\tau\in S_{k+1}$, 
\begin{equation}\label{eqn:transposition}X\cong \Z[\Z/2\Z]\oplus(\Z_-)^{k-2}\text{ as $\Z[\pair{\tau}]$ modules}.\end{equation}

{\it Claim.} There exists a Galois extension $P/\Q$ so that (i) $\Gal(P/\Q)\cong S_{k+1}$ and (ii) under this isomorphism, complex conjugation corresponds to a transposition $\tau\in S_{k+1}$. 

The claim is a special case of the inverse Galois problem. While this problem is open in general, it has been solved for symmetric groups \cite[Prop.\ 2]{kluners-malle}.

Now we use the claim. The claim is equivalent to the existence of a representation 
\[\rho:\Gal(\bar\Q/\Q)\onto S_{k+1}\sbs\GL(X)\] with complex conjugation mapping to $\tau$. The fixed field of $\ker(\rho)$ is the desired Galois extension $P/\Q$. Under the correspondence between $\Z[\Gal(\bar\Q/\Q)]$-modules and algebraic tori (c.f.\ \S\ref{sec:tori}), this implies that there is a torus $S$ with $X(S)\cong X$. 

The $\Q$-rank of $S$ is 0 because $X(S)\cong X$ does not contain a trivial sub-representation. In addition $\rank_\R(S)=1$ by Equation (\ref{eqn:transposition}) and Lemma \ref{lem:rank}. This proves the proposition. 
\end{proof}

Given $k\ge2$, take $S$ as in Proposition \ref{prop:inverse-galois}. By Equation (\ref{eqn:rank}), $\rank S(\Z)=1$. Fix an infinite order element $\la\in S(\Z)$. The surjection $\Z[S_{k+1}]\onto X$ from the proof of Proposition \ref{prop:inverse-galois} induces an embedding $S\hra R_{P/\Q}(G_m)$ where $P/\Q$ is a Galois extension with $\Gal(P/\Q)=S_{k+1}$. Let $T\sbs R_{P/\Q}(G_m)$ be a minimal quasi-split torus in $T=R_{K/\Q}(G_m)$ that contains $S$. The group $\ca O_{K}\rt_\la\Z$ is arithmetic by Proposition \ref{prop:arithmeticity}. In addition the action of $\la$ on $\ca O_K$ is fully irreducible by Lemma \ref{lem:irreducible}. This proves that there is fully irreducible $A\in\GL_{n}(\Z)$ so that $\Ga_A$ is arithmetic, where $n=[K:\Q]=\dim(T)\ge\dim(S)=k$. This proves Theorem \ref{thm:examples}. \qed

\begin{example}
We give an example that demonstrates how Lemma \ref{lem:irreducible} can be further used to determine $n=[K:\Q]$ precisely. 

We show that there is fully irreducible $A\in\GL_{10}(\Z)$ so that $\Ga_A$ is arithmetic. First let $P/\Q$ be a Galois extension with $\Gal(P/\Q)=S_5$ and complex conjugation a transposition, and let $S\sbs R_{P/\Q}(G_m)$ be a torus constructed as in the proof of Theorem \ref{thm:examples}. Choose $K\sbs P$ so that $\Gal(P/K)\cong A_4$. 
As an $S_5$-representation, $\Q[S_5/A_4]$ has character 
\[\begin{array}{c|cccccccc}
\text{conjugacy classes of $S_5$}&()&(12)&(12)(34)&(123)&(1234)&(12345)&(12)(345)\\\hline
\text{character of }\Q[S_5/A_4]&10&0&2&4&0&0&0
\end{array}\]
By computing inner product of characters one finds that $\Q^4\ot\Q_-$ appears in $\Q[S_5/A_4]$ with multiplicity 1. Then $S$ is contained in $T=R_{K/\Q}(G_m)\sbs R_{P/\Q}(G_m)$. We want to show that $T$ is a minimal quasi-split torus containing $S$. We prove this using Lemma \ref{lem:irreducible}(b). 

We want to show that there is no $H\sbs S_5$ so that (i) there is a surjection $\Q[S_5/A_4]\onto\Q[S_5/H]$ and (ii) $\Q^4\ot\Q_-$ is a subrepresentation of $Q[S_5/H]$. Condition (i) implies $H\sbs S_5$ has index $\le10$, since $\Q[S_5/A_4]$ has dimension 10. There is no need to consider the case $[S_5:H]=10$. The subgroups satisfying $[S_5:H]<10$ are $A_5$, $S_4$, and the affine group $C_5\rt C_4$. It is straight-forward to check that $\Q[S_5/H]$ does not contain $\Q^4\ot\Q_-$ as a subrepresentation for each of these subgroups $H$.

\end{example}

\begin{rmk}[Infinitely many examples up to commensurability]
By the result of \cite{kluners-malle} used above, there are infinitely many different $P$ with $\Gal(P/\Q)=S_{k+1}$ and complex conjugation acting as a transposition. Repeating the above construction with different fields $P$ then leads to infinitely many non-commensurable examples (this follows from the discussion of \S\ref{sec:groups}, \ref{sec:arithmeticity}). 

There is another way to obtain infinitely many examples. If $\dim S\ge2$, then there is a surjection $\Z[G]\onto X(S)\oplus X(S)\cong X(S\ti S)$ (here it is relevant that every complex representation of $S_n$ is defined over $\Q$, so each irreducible $\Q[G]$ module appears in $\Q[G]$ with multiplicity equal to its dimension). Then if we repeat the argument above (now taking $T\sbs R_{P/\Q}(G_m)$ a minimal quasi-split torus containing $S\ti S$), the different embeddings $S\hra S\ti S$ lead to non-commensurable examples. Compare with \cite[Prop.\ 5.3]{grunewald-platonov-solvable}. 
\end{rmk}

\subsection{Proof of Theorem \ref{thm:prime-dim}}\label{sec:prime-dim}

Fix an irreducible $A\in\GL_p(\Z)$. Let $\la$ be an eigenvalue of $A$, and set $K=\Q(\la)$. By Proposition \ref{lem:char-poly}, to show that $\Ga_A=\Z^p\rt_A\Z$ is not arithmetic, it suffices to show that $\Ga_\la=\ca O_K\rt_\la\Z$ is not arithmetic. We proceed by considering a series of cases. 

The Galois closure $P$ of $K/\Q$ is either totally real or totally imaginary. 

{\bf Case: $P$ totally real.} In this case, $\Ga_\la$ is non-arithmetic by \cite[proof of Thm.\ 1.3]{grunewald-platonov-polycyclic}. It can also be seen from the proof of Theorem \ref{thm:real}, since the proof shows that if $K$ is totally real and $\ca O_K\rt_\la\Z$ is arithmetic, then $K$ contains a quadratic subfield. But this is impossible if $[K:\Q]$ is an odd prime. 

{\bf Case: $P$ totally imaginary.} The Galois group $G=\Gal(P/\Q)$ is a transitive subgroup of $S_p$. We use two theorems that restrict $G$; see \cite[Thms.\ 1 and 3]{neumann}. 
\begin{itemize}
\item (Burnside) If $G\sbs S_p$ is transitive and not solvable, then $G$ is 2-transitive. 
\item (Galois) If $G\sbs S_p$ is transitive and solvable, then $G\sbs \Z/p\Z\rt(\Z/p\Z)^\ti$. 
\end{itemize} 

In Galois's theorem, the permutation action of $\Z/p\Z\rt(\Z/p\Z)^\ti$ is by affine transformations on $\{0,1,\ldots,p-1\}\cong\Z/p\Z$, i.e.\ given $s\in(\Z/p\Z)^\ti$ and $r\in\Z/p\Z$, define a permutation $x\mapsto sx+r$. 

We consider separately the cases $G$ solvable or not.

{\bf Case: $G$ not solvable.} Set $H=\Gal(P/K)$, and let $\tau\in G\sbs S_p$ be the element that acts by complex conjugation. If $k$ denotes the number of fixed points of $\tau$ acting on $G/H$, then 
\begin{equation}\label{eqn:prime}\Z[G/H]/\Z\cong \Z[\Z/2\Z]^{\fr{p-k}{2}}\oplus\Z^{k-1}\end{equation}
Note that $k\ge1$ because $p=|G/H|$ is odd. To see the isomorphism (\ref{eqn:prime}) concretely, write $G/H=\{y_1,\ldots,y_p\}$ with $y_1,\ldots,y_{p-k}$ permuted in pairs, and $y_{p-k+1},\ldots,y_p$ fixed. Then $[y_1],\ldots,[y_{p-1}]$ forms a basis for $\Z[G/H]/\Z$, and the action of $\tau$ is apparent from this description.

Set $T=R_{K/\Q}^1(G_m)$. By Lemma \ref{lem:rank} and equation (\ref{eqn:prime}), $\rank T(\Z)=\fr{p-k}{2}+(k-1)$. Then $p\ge5$ and $k\ge1$ implies that $\rank T(\Z)\ge2$. Now we use the following lemma. 

\begin{lem}\label{lem:2-transitive}
Let $G$ be any group, and let $H\sbs G$ be a subgroup. If $G$ acts $2$-transitively on $G/H$, then $\C[G/H]\cong\C\oplus V$, and $V$ is irreducible. 
\end{lem}

Lemma \ref{lem:2-transitive} is an exercise in \cite[Exercise 7.2]{serre-reps-finite-groups}, whose solution we sketch. To show $V$ is irreducible, show $\pair{\chi_V,\chi_V}=1$, where $\pair{\cdot,\cdot}$ is the usual inner product on characters. Since $\chi_V=1-\chi_{\C[G/H]}$, this reduces to showing that $\pair{\chi_{\C[G/H]},\chi_{\C[G/H]}}=2$. Note that $\C[G/H]=\Ind_H^G(\C)$ is induced from the trivial representation of $H$, so by Frobenius reciprocity, 
\[\pair{\chi_{\C[G/H]},\chi_{\C[G/H]}}=\pair{\Ind_H^G(1),\chi_{\C[G/H]}}=\pair{1,\Res_H^G \chi_{\C[G/H]}},\]
and the right-hand side is the number of fixed points of $H$ acting on $G/H$. 

The lemma implies that $X(T)\cong\Z[G/H]/\Z$ is irreducible as a $\Z[G]$-module, which implies that $T$ has no nontrivial sub-torus. Hence $\pair{\la}\sbs T$ is Zariski dense for every infinite order $\la\in T(\Z)$. Since $\rank T(\Z)\ge2$, this implies $\ca O_K\rt_\la\Z$ is non-arithmetic by Proposition \ref{prop:arithmeticity}. 

{\bf Case: $G$ solvable.} By Galois's theorem, $G$ is a subgroup of $\Z/p\Z\rt(\Z/p\Z)^\ti$. Write $C_\ell=\Z/\ell\Z$. Observe that $C_p\sbs G$ since $G$ transitive implies that $p$ divides $|G|$ which implies that $G$ has an element of order $p$ (Cauchy's theorem). Then $G$ has the form $C_p\rt H$, where $H\cong C_q$ for some $q$ dividing $p-1$ (recall that $C_p^\ti\cong C_{p-1}$). 

The group $G=C_p\rt H$ has a presentation 
\begin{equation}\label{eqn:presentation}G=\pair{r,s\mid r^p=1=s^q, rs=sr^a},\end{equation}
where $r$ and $s$ generate $C_p$ and $H$ respectively, and $a\in\{1,\ldots,p-1\}\cong C_p^\ti$ is an element of order $q$.

We will argue as in the non-solvable case that $\Z[G/H]/\Z$ is an irreducible $\Z[G]$-module. If $|H|=p-1$ (this is the largest $H$ can be), then $G$ acts 2-transitively on $G/H$ and we can apply Lemma \ref{lem:2-transitive} in the same way. However, if $|H|<p-1$, then $G$ is not 2-transitive, and a different argument is needed. 

\begin{prop}\label{prop:permutation-rep-solvable-case}
Let $G=C_p\rt H$, where $H\sbs C_p^\ti\cong C_{p-1}$ has order $q<p-1$. Then $\Q[G/H]\cong\Q\oplus V$, and $V$ is irreducible.
\end{prop}

\begin{proof}
First we identify $\C[G/H]$ as a $\C[G]$-module. The group $G$ is metacyclic, and its representations over $\C$ are well-known; see e.g.\ \cite[\S47]{curtis-reiner-RTFGAA}. 

The group $G$ has $q+\fr{p-1}{q}$ irreducible representations over $\C$. There are $q$ one-dimensional representations, namely those that factor through the abelianization $G\onto H\cong C_q$. There are $\fr{p-1}{q}$ irreducible $q$-dimensional representations, each induced from an irreducible representation of $C_p$. For $0\le i\le p-1$, let $L_i=\C\{\ell_i\}$ denote the representation of $C_p=\pair{r}$ with action $r(\ell_i)=\ze^i\cdot\ell_i$, where $\ze=e^{2\pi i/p}$ is a primitive $p$-th root of unity. Write 
\[L_i^G:=\C[G]\ot_{\C[H]}L_i\]
for the induced representation. Then $L_i^G$ is irreducible for $1\le i\le p-1$, and $L_1^G,\ldots,L_{(p-1)/q}^G$ are distinct irreducible representations of $G$. 

{\it Claim.} There is an isomorphism of $\C[G]$-modules 
\[\C[G/H]\cong\C\oplus L_1^G\oplus\cdots\oplus L_{(p-1)/q}^G.\]

{\it Discussion of the claim.} The claim can be proved by computing the character of the representation on each side. This is a straightforward computation, so we omit the details. However, below we will use the explicit computation of the character $\chi_i$ of $L_i^G$, so we record it here. 

Fix $1\le i\le\fr{p-1}{q}$. It is easy to compute $\chi_i(s^j)=0$ for each $1\le j\le q-1$. Note that $\{s^jr^i:0\le i\le p-1\}$ forms a conjugacy class (one can show this directly using the presentation (\ref{eqn:presentation})). Then it remains to compute $\chi_i(r^k)$ for $1\le k\le p-1$. Consider the basis $\{1\ot\ell_i, s\ot\ell_i,\ldots,s^{q-1}\ot\ell_i\}$ for  $L_i^G$. In this basis, an easy computation shows that the action of $r\in G$ has matrix 
\[\left(\begin{array}{ccccc}
\ze^i\\&\ze^{ia}\\&&\ze^{ia^2}\\&&&\ddots\\&&&&\ze^{ia^{q-1}}
\end{array}\right)\]
Then $\chi_i(r^k)=(\ze^{ik})+(\ze^{ik})^{a}+\cdots+(\ze^{ik})^{a^{q-1}}$. Observe that $\chi_i(r^k)=\chi_i(r^{ka})$, which agrees with the fact that the conjugacy classes of elements of $\{1,r,\ldots,r^{p-1}\}$ is the same as orbits of the $H$-action on $C_p$. This concludes the discussion of the claim.

With the setup above, we can prove the proposition. From the computation of $\chi_i$, we deduce that the characters $\chi_1,\ldots,\chi_{(p-1)/q}$ for a single orbit under the natural action of $\Gal(\bar\Q/\Q)$. Then the sum of these characters is $\Q$-valued, and there is a smallest $m\ge1$ so that 
\[m(\chi_1+\cdots+\chi_{(p-1)/q})\] 
is realized as a character of $G$-representation defined over $\Q$, and this representation is irreducible over $\Q$. Here $m$ is \emph{Schur index}; see \cite[Cor.\ 10.2]{isaacs}. If $V$ is a complement to $\Q\sbs\Q[G/H]$ (which exists because $\Q[G]$ is semisimple), then $\chi_V=\chi_1+\cdots+\chi_{(p-1)/q}$ (this is the straightforward part of the claim above, whose proof was omitted). This implies that $m=1$ and that $V$ is irreducible. 
\end{proof}

Now we can conclude as in the non-solvable case, replacing Lemma \ref{lem:2-transitive} with Proposition \ref{prop:permutation-rep-solvable-case}: The proposition implies that $T=R_{K/\Q}^1(G_m)$ has no nontrivial sub-torus. Hence to show $\ca O_K\rt_\la\Z$ is non-arithmetic, it suffices to show that $\rank T(\Z)\ge2$, or equivalently (since $\rank_\Q(T)=0$), to show that $\rank_\R(T)\ge2$, c.f.\ Equation (\ref{eqn:rank}). To show $\rank_\R(T)\ge2$ we use Lemma \ref{lem:rank}. 

Let $\tau\in G$ be complex conjugation. Since $P$ is totally imaginary, $\tau$ is nontrivial. Every involution in $G=C_p\rt H$ is conjugate into $H$, and $H\sbs C_{p-1}$ has a unique element of order 2, which acts on $C_p=\pair{r}$ by $r\mapsto r^{-1}$. Then the action of $\tau$ on $G/H=\{eH, rH,\ldots,r^{p-1}H\}$ has a single fixed point, so $\tau$ acts on $\Z[G/H]/\Z$ in the same way that $\Z/2\Z$ acts on $\Z[\Z/2\Z]^{(p-1)/2}$. Hence $X(T)^\tau$ has rank $(p-1)/2\ge2$. This shows $\rank_\R(T)\ge2$, as desired.

This completes the proof of Theorem \ref{thm:prime-dim}.\qed

\begin{rmk}
Note that Theorem \ref{thm:prime-dim} is not true for the primes $p=2,3$. This is obvious for $p=2$. For $p=3$, if $K/\Q$ has 1 real embedding, then the Galois closure $P$ of $K/\Q$ has $\Gal(P/\Q)=S_3$. Then $T=R_{K/\Q}^1(G_m)$ has $\rank T(\Z)=1$, so for any infinite order $\la\in T(\Z)$, the group $\ca O_K\rt_\la\Z$ is arithmetic. Furthermore, the action of $\la$ on $K$ is irreducible by Lemma \ref{lem:irreducible} since there are no intermediate subfields $\Q\sbs K'\sbs K$. 
\end{rmk}

\bibliographystyle{alpha}
\bibliography{ATB-bib}

\begin{thebibliography}{BMR95}

\bibitem[BMR95]{BMR}
B.~H. Bowditch, C.~Maclachlan, and A.~W. Reid.
\newblock Arithmetic hyperbolic surface bundles.
\newblock {\em Math. Ann.}, 302(1):31--60, 1995.

\bibitem[Bor66]{borel-density}
A.~Borel.
\newblock Density and maximality of arithmetic subgroups.
\newblock {\em J. Reine Angew. Math.}, 224:78--89, 1966.

\bibitem[Bor91]{borel-linear-algebraic-groups}
A.~Borel.
\newblock {\em Linear algebraic groups}, volume 126 of {\em Graduate Texts in
  Mathematics}.
\newblock Springer-Verlag, New York, second edition, 1991.

\bibitem[CR06]{curtis-reiner-RTFGAA}
C.~W. Curtis and I.~Reiner.
\newblock {\em Representation theory of finite groups and associative
  algebras}.
\newblock AMS Chelsea Publishing, Providence, RI, 2006.
\newblock Reprint of the 1962 original.

\bibitem[GP98]{grunewald-platonov-polycyclic}
F.~Grunewald and V.~Platonov.
\newblock Non-arithmetic polycyclic groups.
\newblock {\em C. R. Acad. Sci. Paris S\'{e}r. I Math.}, 326(12):1359--1364,
  1998.

\bibitem[GP99]{grunewald-platonov-solvable}
F.~Grunewald and V.~Platonov.
\newblock Solvable arithmetic groups and arithmeticity problems.
\newblock {\em Internat. J. Math.}, 10(3):327--366, 1999.

\bibitem[HKF17]{husert}
D.~Husert, J.~Kl{\"u}ners, and C.~Fieker.
\newblock {\em Similarity of Integer Matrices}.
\newblock Universit{\"a}t Paderborn, Fakult{\"a}t f{\"u}r Elektrotechnik,
  Informatik und Mathematik, 2017.

\bibitem[Isa76]{isaacs}
I.~M. Isaacs.
\newblock {\em Character theory of finite groups}.
\newblock Academic Press [Harcourt Brace Jovanovich, Publishers], New
  York-London, 1976.
\newblock Pure and Applied Mathematics, No. 69.

\bibitem[KM01]{kluners-malle}
J.~Kl\"{u}ners and G.~Malle.
\newblock A database for field extensions of the rationals.
\newblock {\em LMS J. Comput. Math.}, 4:182--196, 2001.

\bibitem[LM33]{latimer-macduffee}
C.~G. Latimer and C.~C. MacDuffee.
\newblock A correspondence between classes of ideals and classes of matrices.
\newblock {\em Ann. of Math. (2)}, 34(2):313--316, 1933.

\bibitem[Mar91]{margulis}
G.~A. Margulis.
\newblock {\em Discrete subgroups of semisimple {L}ie groups}, volume~17 of
  {\em Ergebnisse der Mathematik und ihrer Grenzgebiete (3) [Results in
  Mathematics and Related Areas (3)]}.
\newblock Springer-Verlag, Berlin, 1991.

\bibitem[Neu74]{neumann}
P.~M. Neumann.
\newblock Transitive permutation groups of prime degree.
\newblock In {\em Proceedings of the {S}econd {I}nternational {C}onference on
  the {T}heory of {G}roups ({A}ustralian {N}at. {U}niv., {C}anberra, 1973)},
  pages 520--535. Lecture Notes in Math., Vol. 372, 1974.

\bibitem[PR94]{platonov-rapinchuk}
V.~Platonov and A.~Rapinchuk.
\newblock {\em Algebraic groups and number theory}, volume 139 of {\em Pure and
  Applied Mathematics}.
\newblock Academic Press, Inc., Boston, MA, 1994.
\newblock Translated from the 1991 Russian original by Rachel Rowen.

\bibitem[Ser77]{serre-reps-finite-groups}
J.-P. Serre.
\newblock {\em Linear representations of finite groups}.
\newblock Springer-Verlag, New York-Heidelberg, 1977.
\newblock Translated from the second French edition by Leonard L. Scott,
  Graduate Texts in Mathematics, Vol. 42.

\bibitem[Stu15]{studenmund}
D.~Studenmund.
\newblock Abstract commensurators of lattices in {L}ie groups.
\newblock {\em Comment. Math. Helv.}, 90(2):287--323, 2015.

\bibitem[Wal84]{wallace}
D.~I. Wallace.
\newblock Conjugacy classes of hyperbolic matrices in {${\rm Sl}(n,\,{\bf Z})$}
  and ideal classes in an order.
\newblock {\em Trans. Amer. Math. Soc.}, 283(1):177--184, 1984.

\end{thebibliography}

Department of Mathematics, Brown University\\ 
\emph{Email address}: \texttt{bena\_tshishiku@brown.edu}

\end{document}